\documentclass{cedram-alco}
\usepackage{breakurl}	
\usepackage{url}
\usepackage{xcolor}
\usepackage[ps,dvips,color,all]{xy}
\usepackage{upgreek}
\usepackage{pstricks}
% ...
\title
[A spectral theory for combinatorial dynamics]
{A spectral theory for combinatorial dynamics}

\author[\initial{J.} Propp]{\firstname{James} \lastname{Propp}}
\address{University of Massachusetts Lowell\\ Department of Mathematical Sciences}

\urladdr{https://jamespropp.org}


\subjclass{05E18, 06A07; need to fix this}

\newcommand{\cA}{{\mathcal{A}}}

\newcommand{\cO}{{\mathcal{O}}}
\newcommand{\cP}{{\mathcal{P}}}

\newcommand{\rowP}{\rho_{\cP}}

\newcommand{\C}{\mathbb{C}}

\newcommand{\R}{\mathbb{R}}

\newcommand{\downt}{\nabla}

\newcommand{\upti}{\Delta^{-1}}

\newcommand{\nh}{\ar@{-}[ru]}
\newcommand{\sh}{\ar@{-}[rd]}

\newcommand{\zb}{\bar{\zeta}}
\newcommand{\z}{{\zeta}}

\begin{document}
%% Abstracts must be placed before \maketitle

\begin{center}
{\it To the memories of Jacob Feldman and Gian-Carlo Rota}
\end{center}

\begin{abstract}
This article proposes a framework for the study of
periodic maps $T$ from a (typically finite) set $X$ to itself
when the set $X$ is equipped with one or more real- or complex-valued functions.
The main idea, inspired by 
the time-evolution operator construction from ergodic theory,
is the introduction of a vector space that contains the given functions
and is closed under composition with $T$, 
along with a time-evolution operator on that vector space.
I show that the invariant functions and 0-mesic functions 
span complementary subspaces associated respectively with 
the eigenvalue 1 and the other eigenvalues.
Alongside other examples, 
I give an explicit description of the spectrum 
of the evolution operator when $X$ is the set of 
$k$-element multisets with elements in $\{0,1,\dots,n-1\}$,
$T$ increments each element of a multiset by 1 mod $n$,
and $g_i: X \rightarrow \R$ (with $1 \leq i \leq k$)
maps a multiset to its $i$th smallest element.
% We also present a conjecture related to rowmotion in a product of two chains.
\end{abstract}

\maketitle

\begin{section}{Introduction}
\label{sec:intro}

Recent work in dynamical algebraic combinatorics
has paid a great deal of attention to {\em homomesies:}
numerical functions on a dynamical system with the property that
the average of the function over an orbit 
doesn't depend on which orbit one takes.
Missing from this work is attention to {\em invariants:}
quantities that are constant on orbits.
It would be conceptually helpful to bring 
homomesies and invariants into a uniform framework. 
The article [PR] that introduced the concept of homomesy
provided this framework in the special case 
of linear actions (see section 2.4 of that article),
but most actions of interest to combinatorialists
are nonlinear (piecewise linear, birational, or purely combinatorial).
Here I introduce a vector space that may be
the proper setting for a unified treatment of homomesies and invariants;
a periodic action induces a diagonalizable linear map on the vector space,
and the spectrum of the map carries dynamical information
about homomesies and invariants for the action and its powers.

\end{section}

\begin{section}{Linearization}
\label{sec:framework}

Given a set $X$, an invertible map $T$ from $X$ to itself
satisfying $T^n = {\rm Id}_X$ 
(that is, $T^n(x) = x$ for all $x \in X$) with $n \geq 1$,
and a collection of functions $g_1, \dots, g_k$ from $X$ to $\C$,
% (or to some other field of characteristic zero),
let $V$ be the linear span of all functions of the form
$g_i \circ T^j$ with $1 \leq i \leq k$, $0 \leq j < n$.
I will sometimes refer to the $g_i$'s as {\em statistics}
(or sometimes as the ``original'' statistics
as opposed to the time-shifted statistics $g_i \circ T^j$).
$V$ is a finite-dimensional space of dimension at most $kn$.

The role played by the $g_i$'s is crucial,
much as the choice of an algebra of measurable sets
is crucial in ergodic theory.
When we enlarge our set of initial statistics we potentially enlarge $V$,
and when we restrict our set of initial statistics we potentially reduce $V$.
As an extreme case, if our collection of statistics is empty,
$V$ is 0-dimensional regardless of the dynamics of the map $T$.

For all $f \in V$ define $Uf \in V$ by $(Uf)(x) = f(T(x))$.
We will call the action of $U$ on $V$
the {\em linearization} of the original action of $T$ on $X$.
It is easy to check that $U$ sends $V$ to itself
(e.g., $U$ sends $g_i \circ T^j$ to $g_i \circ T^{j+1}$,
which in the case $j=n-1$ is $g_i$);
that $U$ is linear; that $U^n$ is the identity map on $V$;
and that $U$ has inverse $U^{n-1}$.
$V$ is the {\em dynamical span\/} of $g_1,\dots,g_k$
in the sense that it is the smallest vector space
that contains $g_1,\dots,g_k$ and is closed under the action of $U$.
We will also call $V$ the {\em dynamical closure\/}
of the linear span of $g_1,\dots,g_k$.
The function $g_i \circ T^j$ can also be written as $U^j g_i$.

Let us say that a function $f \in V$ is {\em homomesic}
if $\frac{1}{n} \sum_{j=0}^{n-1} U^j f$ is a constant function,
and more specifically {\em c-mesic} 
if $\frac{1}{n} \sum_{j=0}^{n-1} U^j f = c$ for all $x \in X$,
that is, if $\frac{1}{n} \sum_{j=0}^{n-1} f (T^j x) = c$ for all $x \in X$.
Let us say that a function $f \in V$ is {\em invariant}
if $U f = f$, that is, if $f(Tx) = f(x)$ for all $x \in X$.

% Break this out as a definition
When $X$ is finite (as it will be 
except in sections~\ref{sec:rowmotion} and~\ref{sec:lyness})
we create an $|X|$-by-$kn$ matrix $M$ whose rows correspond to elements $x \in X$
(the order in which the elements of $X$ are listed is unimportant)
and whose columns from left to right correspond to the respective functions
$g_1,\dots,g_k,Ug_1,\dots,Ug_k,\dots,U^{n-1}g_1,\dots,U^{n-1}g_k$,
where the entry in the $x$ row and the $U^j g_i$ column 
is $(U^j g_i)(x) = g_i(T^j x)$;
call $M$ the {\em presenting matrix} of $U$, 
and note that $V$ can be identified with the span of the columns of $M$.
In particular, $\dim V$ is the rank of $M$.
% Is there a linear map associated with the matrix?

Since $U^n = I$ (the identity map from $V$ to itself),
$U$ is a diagonalizable operator on $V$ 
whose eigenvalues are $n$th roots of unity;
there exists a basis for $V$ whose elements are eigenvectors for $U$.
(Typically our statistics are real-valued, but if we want to look 
at the eigenspaces associated with eigenvalues other than 1 and $-1$
we need to treat $V$ as a vector space over $\C$.)
Let $V_1$ be the span of the 1-eigenvectors
(i.e., the nonzero elements $f \in V$ satisfying $Uf=1f=f$)
and let $V_1^\perp$ be the span of the other (non-unital) eigenvectors,
so that $V = V_1 \oplus V_1^\perp$ and $\dim V$ = $\dim V_1 + \dim V_1^\perp$.
(Note that despite the notation no inner product is involved;
that is, $V_1^\perp$ is a complement but not an orthocomplement.)

\begin{prop}
$V_1$ is the space of invariants and $V_1^\perp$ is the space of 0-mesies.
\end{prop}

\begin{proof}
If $f \in V_1$ then $Uf=1f=f$ and vice versa; 
hence $V_1$ is the space of invariants.
If $f \in V_1^\perp$ so that $f$ is in the span of the non-unital eigenvectors
then $f+Uf+U^2f+\cdots+U^{n-1}f = 0$ (since in the case where 
$f$ is a $\zeta$-eigenvector with $\zeta \neq 1$
we have $f+Uf+U^2f+\cdots+U^{n-1}f = 
(1+\zeta+\zeta^2+\cdots+\zeta^{n-1})f = 0f$).
Conversely, given $f$ satisfying $f+Uf+U^2f+\cdots+U^{n-1}f = 0$,
if we write $f = f_1 + f_2$ with $f_1 \in V_1$ and $f_2 \in V_1^\perp$
and we apply $I+U+U^2+\cdots+U^{n-1}$ to $f_1+f_2$,
we get $nf_1$, which can only vanish if $f_1$ does,
implying $f \in V_1^\perp$.
\end{proof}

It follows that the multiplicity of 1 as an eigenvalue of $U$
gives the number of linearly independent invariants in $V$
while the sum of the multiplicities of the non-unital eigenvalues of $U$
gives the number of linearly independent 0-mesies in $V$.
It is easy to show that every homomesic function $f$ can be written uniquely 
as the sum of a constant function and a 0-mesic function, specifically 
the constant function $g = (1/n) (f + U f + U^2 f + ... + U^{n-1} f)$ 
and the 0-mesic function $h = f - g$.

In view of the complementary nature of $V_1$ and $V_1^\perp$,
it might be appropriate to refer to the elements of $V_1^\perp$
as ``coinvariants'' rather than 0-mesies,
but this word is already in use with a different meaning.
Another term for elements of $V_1^\perp$ that seems apt is ``survariants'',
which has no existing meaning and in some ways seems preferable to ``0-mesies'';
however, the terms ``homomesy'' and ``homomesic''
seem to have been adopted to the point where
a change in nomenclature might be confusing.

It is possible for $V_1$ to be trivial;
for instance, if $X = \{1,-1\}$ with $T(x)=-x$ of order $n=2$
equipped with the identity statistic $g_1(x) = x$,
then $V$ is 1-dimensional with $\dim V_1 = 0$ and $\dim V_1^\perp = 1$.
However, for all examples considered in this paper
$V_1$ will contain the constant functions and hence have dimension $\geq 1$.
As long as there is at least one homomesic function $f$ that is not a 0-mesic,
the nonzero constant function $\sum_{j=0}^{n-1} U^j f$ is in $V$,
so that the constant functions form a 1-dimensional subspace of $V_1$.

I proceed to give alternative characterizations of $V_1$ and $V_1^\perp$.

\begin{prop}
\label{inv-rep}
$f \in V_1$ if and only if there exists
$g \in V$ with $f = g + Ug + U^2g + \cdots + U^{n-1}g$.
\end{prop}

\begin{proof}
If $f = g + Ug + U^2g + \cdots + U^{n-1}g$
then $Uf = Ug + U^2g + U^3g + \cdots + U^ng = f$
(since $U^ng=g$), so $f \in V_1$.
Conversely, suppose $f$ is in $V_1$.
Then putting $g=(1/n)f$ we have
$f = g + Ug + U^2g + \cdots + U^{n-1}g$.
\end{proof}

Consequently, the space of invariants is spanned by the sums 
$f_i := g_i + Ug_i + \cdots + U^{n-1}g_i$;
call these the {\em spanning invariants}.
Since these $k$ functions span $V_1$, $\dim V_1 \leq k$.
If we divide the presenting matrix into $n$ blocks of width $k$
and sum those blocks, we obtain a matrix $M_1$ whose column span is $V_1$.

I next generalize Proposition~\ref{inv-rep} to eigenfunctions,
exploiting the discrete Fourier transform.
Given $\zeta \in \C$ with $\zeta^n = 1$,
let $V_\zeta$ be the space of $f \in V$ with $Uf = \zeta f$
(or equivalently $\zb U f = f$),
so that $V_1^\perp$ is the direct sum
of the spaces $V_\zeta$ with $\zeta \neq 1$.

\begin{prop}
\label{eigen}
For $\zeta^n = 1$, $f \in V_\zeta$ if and only if there exists
$g \in V$ with $f = g + \zb U g + \cdots + \zb^{n-1} U^{n-1} g$.
\end{prop}

\begin{proof}
The proof of Proposition~\ref{inv-rep} applies to
any linear operator $U'$ on the $\C$-vector space $V$
that satisfies $(U')^n = I$, and in particular 
applies to $U' = \overline{\zeta} U$.
\end{proof}

Consequently, the $\zeta$-eigenspace is spanned by the functions
$g_i + \zb U g_i + \cdots + \zb^{n-1} U^{n-1} g_i$;
let us call these the {\em spanning $\zeta$-eigenfunctions}.
We have $\dim V_\zeta \leq k$.
If we take the blocks discussed following the proof of Proposition~\ref{inv-rep}
and sum them with respective coefficients $1,\zb,\zb^2,\dots$,
we obtain a matrix $M_\zeta$ whose column span is $V_\zeta$.

\medskip

I provide an alternate characterization of the elements of 
$V_1^\perp = \bigoplus_{\zeta^n=1,\ \zeta \neq 1} V_\zeta$:

\begin{prop}
\label{mesic}
$f \in V_1^\perp$ if and only if there exists $g \in V$ with $f = g - Ug$.
\end{prop}
\noindent
(Note that we could just as well have used $Ug - g$ as $g$.
A function of the form $Ug - g$ is called 
a coboundary in dynamical system theory; 
it measures the increase in $g$ from one moment to the next.)

\begin{proof}
If $f = g - Ug$, then 
$f+Uf+\cdots+U^{n-1}f = (g-Ug)+(Ug-U^2g)+\cdots+(U^{n-1}g-U^ng) 
= g-U^ng = 0$ (since $U^n g = g$),
so $f$ is 0-mesic and belongs to $V_1^\perp$.
Conversely, suppose $f$ is in $V_1^\perp$.
Let $h = 1f + 2Uf + 3U^2f + \cdots + nU^{n-1}f$.  Then 
\begin{eqnarray*}
h-Uh & = & \ \ (1U^0f + 2U^1f + 3U^2f + \cdots + (n-1)U^{n-2}f + nU^{n-1}f) \\
& & - (1U^1f + 2U^2f + 3U^3f + \cdots + (n-1)U^{n-1}f + nU^nf) \\
& = & (U^0f + U^1f + U^2f + \cdots + U^{n-1}f) - nU^nf \\
& = & 0 - nf \ \ \mbox{(because $f$ is 0-mesic and $U^n f = f$)} \\
& = & -nf,
\end{eqnarray*}
so putting $g = -(1/n)h$ we have $f = g-Ug$.  \end{proof}

If we are interested in counting all nonunital eigenvalues (with multiplicity),
we can just take the rank of the matrix $M-M'$
where $M$ is the presenting matrix and
$M'$ is obtained from $M$ by cyclically shifting
the columns $k$ positions to the right.

It should be stressed that in the framework being proposed here,
 0-mesies can be seen either as coboundaries of arbitrary functions 
or as combinations of eigenfunctions of the time-evolution operator $U$.
In particular, real 0-mesies can occur as
combinations of complex eigenfunctions.

\medskip

It may be helpful to note that the 0-mesies form 
the kernel of $I+U+U^2+\cdots+U^{n-1}$,
which Proposition~\ref{mesic} identifies as the image of $I-U$;
likewise, the invariants form the kernel of $I-U$,
which Proposition~\ref{inv-rep} identifies 
as the image of $I+U+U^2+\cdots+U^{n-1}$;
and more generally the $\zeta$-eigenfunctions form the kernel of $I-\zb U$,
which Proposition~\ref{eigen} identifies as the image 
of $I+\zb U + \zb^2 U^2 + \cdots + \zb^{n-1} U^{n-1}$.
In each case we have a short exact sequence of vector spaces.
Two function $g_1,g_2$ give rise to the same element of $V_1$ 
via Proposition~\ref{inv-rep} iff they differ by an element of $V_1^\perp$,
and two function $g_1,g_2$ give rise to the same element of $V_1^\perp$ 
via Proposition~\ref{mesic} iff they differ by an element of $V_1$.

Let us define $V^{(0)}$
% the space of {\em time-0 0-mesies\/}
as the intersection of $V_1^\perp$ with the span of $g_1,\dots,g_k$.
There is a sense in which $V^{(0)}$ determines $V_1^\perp$,
namely, the linear combination $\sum_{i,j} a_{i,j} \,U^j g_i$ is in $V_1^\perp$
if and only if the linear combination $\sum_{i,j} a_{i,j} g_i$ 
(in which $U^j g_i$ is replaced by $g_i$) is in $V_1^\perp$.
That is because every function of the form $U^j g - g$ is a coboundary 
(note that $U^j g - g = U(g+Ug+\dots+U^{j-1}g) - (g+Ug+\dots+U^{j-1}g)$).

The sum of the multiplicities of the non-unital eigenvalues
(that is, the eigenvalues unequal to 1)
gives us the dimension of the space of 0-mesies.
The dynamical significance of the specific multiplicities 
of individual non-unital eigenvalues is subtler.
For every $d$ dividing $n$, the sum of the multiplicities 
of the eigenvalues $\zeta$ satisfying $\zeta^d = 1$
is the dimension of the space of invariants of $T^d$,
while the sum of the multiplicities of those $\zeta$ with $\zeta^d \neq 1$
is the dimension of the space of 0-mesies of $T^d$.
Additional meaning of the multiplicities appears if,
in the spirit of classical invariant theory, one extends $V$ to a ring $R$,
introducing statistics that are products of the $g_i$'s.
Invariant functions in the ring $R$
can arise from noninvariant eigenfunctions in $V$
associated with complex eigenvalues whose product is 1;
see the paragraph following the proof of Proposition~\ref{prop:two}.
However, it should be noted that the ring $R$ 
need not be graded by polynomial degree;
e.g., in the example treated in Proposition~\ref{prop:two},
each $g_i$ takes values in $\{0,1\}$ and hence satisfies $g_i^2 = g_i$.
Also note that when the set $X$ is finite, the ring $R$ 
(being a set of functions with domain $X$),
viewed as a vector space, must be finite-dimensional.

\end{section}

\begin{section}{Example: Rowmotion in a Chain}

\label{sec:chain}

Fix $n,k \geq 2$
and let $X_{n,k}$ be the set of $k$-element multisets 
with elements belonging to $\{0,1,\dots,n-1\}$.
Let us denote a generic element of $X_{n,k}$ 
by $x = (x_1,x_2,\dots,x_k)$
with $x_1 \leq x_2 \leq \cdots \leq x_k$.
We can represent an element of $X_{n,k}$
by writing down its elements with appropriate multiplicities,
in weakly increasing order, 
with parentheses and intervening commas omitted.
For instance, $X_{3,3} = \{000, 001, 002, 011, 012, 022, 111, 112, 122, 222\}$.
I will use superscripts to indicate repetition,
e.g., I write $X_{n,k} = \{0^k,0^{k-1}1,\dots,(n-1)^k\}$.
It is well-known that $|X_{n,k}| = {n+k-1 \choose k}$.

Given $x = (x_1,x_2,\dots,x_k)$, put $x_0 = 0$ and $x_{k+1} = n-1$,
and for $1 \leq i \leq k$ define 
the reflections $\rho_i: X_{n,k} \rightarrow X_{n,k}$
by $\rho_i(x) = x'$ where $x'_i = x_{i-1}+x_{i+1}-x_i$
and $x'_j = x_j$ for all $j \neq i$.
The $\rho_i$'s satisfy the Coxeter relations
$\rho_i^2 = 1$, $(\rho_i \rho_{i+1})^3 = 1$,
and $(\rho_i \rho_j)^2 = 1$ for $|i-j| > 1$,
and thus form a representation of the Coxeter group $A_k$.

Any product of all the $\rho_i$'s, each taken one at a time,
is a Coxeter element $\gamma$ satisfying $\gamma^{k+1} = 1$.
For convenience, we take $\gamma_{n,k} = \rho_k \circ \cdots \circ \rho_1$
which updates entries from left to right.
For instance, $\gamma_{3,3}$ sends 000 to 002 to 022 to 222 to 000
and sends 001 to 012 to 122 to 111 to 001
and sends 011 to 112 to 011.
We have $\gamma_{n,k}(x) = x'$ where $x'_i = x'_{i-1}+x_{i+1}-x_i$;
that is, the new value in location $i$ equals the new value in location $i-1$
plus the old value in location $i+1$ minus the old value in location $i$.
The operation $\gamma$ can be seen as 
a special case of piecwise-linear rowmotion on 
the order polytope of a product of two chains
in the case where one of the chains is of length 1.
For more on piecewise-linear rowmotion, see [EP2]
and section~\ref{sec:rowmotion} of this article.

For $x = (x_1,x_2,\dots,x_k)$ and $0 \leq i \leq k+1$ let $g_i(x) = x_i$,
so that for instance $g_1(001) = 0$, $g_2(001) = 0$, and $g_3(001) = 1$.
The maps $g_i: X \rightarrow \R$ ($1 \leq i \leq k$) and $T: X \rightarrow X$
are all linear with respect to $x_0,\dots,x_{k+1}$ 
so the dynamical span $V$ of $g_1,\dots,g_k$ is easy to determine;
it is spanned by $g_1,\dots,g_k$ along with the constant functions.
Define $U$ as above.
% Refer to a specific definition?
The relation $x'_i = x'_{i-1}+x_{i+1}-x_i$
gives $Ug_i = Ug_{i-1} + g_{i+1} - g_i$.
If we apply $U^j$ to the preceding relation
and sum as $j$ goes from $0$ to $k$,
we find that the orbit-averages of $g_{i-1}$, $g_i$, and $g_{i+1}$
are in arithmetic progression.
Since the orbit-averages of $g_0$ and $g_{k+1}$ 
are 0 and $n-1$ respectively,
we see that $g_i$ is $c_i$-mesic with $c_i = i\frac{n-1}{k+1}$.
The $k$ functions $g_i-c_i$ ($1 \leq i \leq k$) 
are the 0-mesies of this action
while the constant functions are the only invariants;
together they span the full $(k+1)$-dimensional vector space
dynamically spanned by the $g_i$'s.
The action of $U$ on this space is a simple rotation whose spectrum 
assigns multiplicity 1 to each $k+1$st root of 1.
Concretely, $U$ acts as a cyclic rotation on the functions
$g_1-g_0$, $g_2-g_1$, \dots, $g_{k+1}-g_{k}$.

\end{section}

\begin{section}{Example: Multiset Rotation} 

\label{sec:multiset}

Fix $n,k \geq 2$, and define $X_{n,k}$ as in section~\ref{sec:chain}.

Let us define the {\em rotation operator} $T_{n,k}: X_{n,k} \rightarrow X_{n,k}$
that increments each element of $x$ by 1 mod $n$,
sending $i$ to $i+1$ for $i<n-1$ and sending $n-1$ to 0;
for instance, $T_{3,3}(001) = 112$ and $T_{3,3}(112) = 022$
(note that the 2 in 112 has become a 0 and has moved to the left).
It is obvious that $T_{n,k}^n$ is the identity on $X_{n,k}$.
In contexts where it is safe to do so without confusion,
I will omit subscripts on $X$ and $T$.
Let $g_i(x)$ denote the $i$th smallest element of $x$ as before,
and let $U_{n,k}$ denote the time-evolution operator
for the action $f \mapsto f \circ T_{n,k}$
on the dynamical closure of the linear span of $g_1,\dots,g_k$.

It has been known for several years,
as part of the unwritten lore of dynamical algebraic combinatorics,
that $g_i + g_j$ takes average value $n-1$ on each orbit of $T$
whenever $i+j=k+1$ (including the case $i=j=(k+1)/2$ when $k$ is odd).
When we pass to the dynamical span
we find the invariants that were ``missing''
from earlier treatments of this example.

The case $n=2$ behaves differently than the general case,
so we give it separate consideration.
Here the possible eigenvalues are $1$ and $-1$.

\begin{prop}
\label{prop:two}
The multiplicities of the eigenvalues $1$ and $-1$ 
in the spectrum of $U_{2,k}$
are $\lceil (k+1)/2 \rceil$ and $\lfloor (k+1)/2 \rfloor$ respectively.
\end{prop}

\begin{proof}
$X$ has $k+1$ elements;
if $k$ is odd, $X$ consists of $(k+1)/2$ orbits of size 2,
and if $k$ is even, $X$ consists of $k/2$ orbits of size 2 
and the fixed point $0^{k/2} 1^{k/2}$.
Here I show the presenting matrix for $k=3$
whose columns correspond to the six functions
$g_1$, $g_2$, $g_3$, $Ug_1$, $Ug_2$, and $Ug_3$,
and the presenting matrix for $k=4$
whose columns correspond to the eight functions
$g_1$, $g_2$, $g_3$, $g_4$, $Ug_1$, $Ug_2$, $Ug_3$, and $Ug_4$:
$$ \left(
\begin{array}{ccccccc}
\,0\,&\,0\,&\,0\,&\,&\,1\,&\,1\,&\,1 \\
0&0&1&&0&1&1 \\
0&1&1&&0&0&1 \\
1&1&1&&0&0&0 
\end{array}
\right) $$
$$ \left(
\begin{array}{ccccccccc}
\,0\,&\,0\,&\,0\,&\,0\,&\,&\,1\,&\,1\,&\,1\,&\,1 \\
0&0&0&1&&0&1&1&1 \\
0&0&1&1&&0&0&1&1 \\
0&1&1&1&&0&0&0&1 \\
1&1&1&1&&0&0&0&0 
\end{array}
\right) $$

It is easy to see that for general $k$
the $k+1$ rows of the presenting matrix are linearly independent,
so that the matrix has rank $k+1$,
which implies that the dynamical span of the original vector space
(the column-span of the matrix) is $(k+1)$-dimensional.
If we add the left half of the $(k+1)$-by-$2k$ presenting matrix 
to the right half,
we get a matrix whose columns correspond to 
the spanning invariants $g_1+Ug_1$, $g_2+Ug_2,\dots$
and therefore span the space of invariants.
In the cases $k=3$ and $k=4$, these matrices are
$$ \left(
\begin{array}{ccc}
\,0\,&\,0\,&\,0 \\
0&0&1 \\
0&1&1 \\
1&1&1 
\end{array}
\right)
+
\left(
\begin{array}{ccc}
\,1\,&\,1\,&\,1 \\
0&1&1 \\
0&0&1 \\
0&0&0 
\end{array}
\right)
=
\left(
\begin{array}{ccc}
\,1\,&\,1\,&\,1 \\
0&1&2 \\
0&1&2 \\
1&1&1
\end{array}
\right)
$$
and
$$
\left(
\begin{array}{cccc}
\,0\,&\,0\,&\,0\,&\,0 \\
0&0&0&1 \\
0&0&1&1 \\
0&1&1&1 \\
1&1&1&1 
\end{array}
\right)
+
\left(
\begin{array}{cccc}
\,1\,&\,1\,&\,1\,&\,1 \\
0&1&1&1 \\
0&0&1&1 \\
0&0&0&1 \\
0&0&0&0 
\end{array}
\right)
=
\left(
\begin{array}{cccc}
\,1\,&\,1\,&\,1\,&\,1 \\
0&1&1&2 \\
0&0&2&2 \\
0&1&1&2 \\
1&1&1&1 
\end{array}
\right)
$$
For general $k$, the first $\lceil (k+1)/2 \rceil$ rows of the matrix
are linearly independent and the remaining rows repeat earlier rows,
so the rank is $\lceil (k+1)/2 \rceil$, 
so $\dim V_1 = \lceil (k+1)/2 \rceil$ as claimed.
To find $\dim V_1^\perp = \dim V_{-1}$,
we subtract the two half-matrices instead of adding, obtaining 
$$ \left(
\begin{array}{ccc}
\,0\,&\,0\,&\,0 \\
0&0&1 \\
0&1&1 \\
1&1&1 
\end{array}
\right)
-
\left(
\begin{array}{ccc}
\,1\,&\,1\,&\,1 \\
0&1&1 \\
0&0&1 \\
0&0&0 
\end{array}
\right)
=
\left(
\begin{array}{rrr}
-1&-1&-1 \\
 0&-1& 0 \\
 0& 1& 0 \\
 1& 1& 1
\end{array}
\right)
$$
and
$$
\left(
\begin{array}{cccc}
\,0\,&\,0\,&\,0\,&\,0 \\
0&0&0&1 \\
0&0&1&1 \\
0&1&1&1 \\
1&1&1&1 
\end{array}
\right)
-
\left(
\begin{array}{cccc}
\,1\,&\,1\,&\,1\,&\,1 \\
0&1&1&1 \\
0&0&1&1 \\
0&0&0&1 \\
0&0&0&0 
\end{array}
\right)
=
\left(
\begin{array}{rrrr}
-1&-1&-1&-1 \\
 0&-1&-1& 0 \\
 0& 0& 0& 0 \\
 0& 1& 1& 0 \\
 1& 1& 1& 1 
\end{array}
\right)
$$
in the cases $k=3$ and $k=4$.
The first $\lfloor (k+1)/2 \rfloor$ rows of the matrix
are linearly independent and each of the remaining rows
is either the zero vector or the negative of an earlier row,
so the rank is $\lfloor (k+1)/2 \rfloor$, 
so $\dim V_{-1} = \lfloor (k+1)/2 \rfloor$ as claimed.
\end{proof}

Following up on my earlier remark about invariant theory
at the end of section~\ref{sec:framework}, let us pause to consider 
the quadratic invariant $(g_1 + ... + g_k)(Ug_1 + ... + Ug_k)$, 
equal to the number of 1's in $x$ times the number of 1's in $Tx$
(or, equivalently,
the number of 0's in $x$ times the number of 1's in $x$).
If for all $i$ we put $s_i = (g_i + Ug_i)/2$ 
(an eigenfunction for the eigenvalue 1)
and $a_i = (g_i - Ug_i)/2$ (an eigenfunction for the eigenvalue $-1$)
so that $g_i = s_i + a_i$ and $Ug_i = s_i - a_i$, then putting
$S = s_1 + ... + s_k$ and $A = a_1 + ... + a_k$ we see that
the quadratic invariant equals $(S+A)(S-A) = S^2 - A^2$.
% (Here ``$s$'' stands for ``symmetric''
% and ``$a$'' stands for ``antisymmetric'' or ``alternating''.)
This representation shows us explicitly how the quadratic invariant arises
from all the possible pairs of eigenfunction $s_i,s_j$ with the eigenvalue $1$
and all the possible pairs of eigenfunction $a_i,a_j$ with the eigenvalue $-1$.

\medskip

We now consider $n>2$.

\begin{prop}
For $n>2$, the multiplicity of the eigenvalue $\zeta$
in the spectrum of $U_{n,k}$
is $\lfloor k/2 \rfloor + 1$ when $\zeta = 1$
and $k$ when $\zeta \neq 1$.
\end{prop}

\begin{proof}
I prove the second claim first.  
Given $\zeta \neq 1$ with $\zeta^n = 1$, define 
$h_i = g_i + \zb U g_i + \zb^{2} U^2 g_i + \cdots + \zb^{n-1} U^{n-1} g_i$
for $1 \leq i \leq k$ 
(the spanning $\zeta$-eigenfunctions; see Proposition~\ref{eigen}).
Since these functions span $V_\zeta$, $\dim V_\zeta \leq k$.
To prove the reverse inequality and thereby prove equality
it suffices to find a $k$-element subset of $X$, 
call it $Y = \{y^{(1)},\dots,y^{(k)}\}$, so that the vectors 
$(g_i(y^{(1)}),\dots,g_i(y^{(k)}))^t$ (with $1 \leq i \leq k$)
are linearly independent; that is, 
we need the $k$-by-$k$ matrix $M$ 
whose $i,j$th entry is $g_i(y^{(j)})$ to be nonsingular.
Let $y^{(j)} = 0^{k-j+1} 1^{j-1}$.
Then $M$ has only four distinct entries:
\begin{eqnarray*}
W & = & 0 + 1\z + 2\z^2 + \cdots + (n-2)\z^{n-2} + 0\z^{n-1} \\
N & = & 0 + 1\z + 2\z^2 + \cdots + (n-2)\z^{n-2} + (n-1)\z^{n-1} \\
S & = & 1 + 2\z + 3\z^2 + \cdots + (n-1)\z^{n-2} + 0\z^{n-1} \\
E & = & 1 + 2\z + 3\z^2 + \cdots + (n-1)\z^{n-2} + (n-1)\z^{n-1} 
\end{eqnarray*}
(For later use, note that
\begin{eqnarray*}
(1-\z)(S-N) 
& = & (1 - \z) (1 + \z + \z^2 + \cdots + \z^{n-2} - (n-1)\z^{n-1}) \\
& = & 1 - \z^{n-1} - (n-1)\z^{n-1} + (n-1) \z^n \\
% & = & 1 - n\z^{n-1} + (n-1) \\
& = & n(1 - \z^{n-1})
\end{eqnarray*}
which is nonvanishing for all $n > 1$ and that
\begin{eqnarray*}
(1-\z)(E-W) 
& = & (1 - \z) (1 + \z + \z^2 + \cdots + \z^{n-2} + (n-1)\z^{n-1}) \\
& = & 1 - \z^{n-1} + (n-1)\z^{n-1} - (n-1) \z^n \\
% & = & 1 + (n-2)\z^{n-1} - (n-1) \\
& = & (n-2)(\z^{n-1} - 1) 
\end{eqnarray*}
which vanishes for $n=2$ and is nonvanishing for $n > 2$.)

When $k$ is odd, 
$M$ has the form illustrated below for $k=3$ and $k=5$.
$$
\left(
\begin{array}{ccc}
N&N&N \\
W&N&E \\
W&S&E
\end{array} \right)
\ \ \ \ 
\left(
\begin{array}{ccccc}
N&N&N&N&N \\
W&N&N&N&E \\
W&W&N&E&E \\
W&W&S&E&E \\
W&S&S&S&E
\end{array} \right)
$$
Note that each matrix is split by diagonal and antidiagonal lines
into four zones, each of which has all entries equal. 
Similarly, when $k$ is even, 
$M$ has the form illustrated below for $k=4$ and $k=6$.
$$
\left(
\begin{array}{cccc}
N&N&N&N \\
W&N&N&E \\
W&W&E&E \\
W&S&S&E
\end{array}
\right)
\ \ \ \ 
\left(
\begin{array}{cccccc}
N&N&N&N&N&N \\
W&N&N&N&N&E \\
W&W&N&N&E&E \\
W&W&W&E&E&E \\
W&W&S&S&E&E \\
W&S&S&S&S&E
\end{array} \right)
$$
Let $D_k(N,E,S,W)$ denote the determinant 
of the $k$-by-$k$ matrix of the above form,
where for now $N,E,S,W$ are to be treated as formal indeterminates.
There are many ways to evaluate these determinants,
but my favorite is the following argument, communicated to me by Joe Buhler.
If we subtract $W/E$ times the last column from the first column,
so that the only nonvanishing entry in the first column
is $N - (E/W)N = N(E-W)/E$, we obtain the recurrence
$$D_k(N,E,S,W) = (N(E-W)/E) \ D_{k-1}(E,S,W,N)$$
(where the cyclic rotation of the arguments
corresponds to 90 degree rotation of the submatrix).
With this recurrence and the initial condition $D_2(N,E,S,W) = N(E-W)$
it is easy to prove the general formula
$$
D_{2i+r}(N,E,S,W) = 
\left\{
\begin{array}{ll}
(-1)^i N (N-S)^{i-1} (W-E)^{i} & \mbox{if $r=0$,} \\
(-1)^i N (N-S)^{i} (W-E)^{i} & \mbox{if $r=1$} 
\end{array}
\right.
$$
by induction.
When $\zeta \neq 1$ we have $N - S \neq 0$ and $W - E \neq 0$
as noted earlier so $D_k(N,E,S,W) \neq 0$,
proving that the matrix is nonsingular as claimed.
(This is the part of the proof that assumes $n>2$; we have $W=E$ when $n=2$.)

It remains to consider $\zeta=1$.
In this case we have $N=S$ so the matrix is singular.
Indeed, the last $\lfloor (k-1)/2 \rfloor$ rows of $M$
coincide with earlier rows,
showing that the corank is at least $\lfloor (k-1)/2 \rfloor$,
or equivalently that the rank is at most 
$\lceil (k+1)/2 \rceil = \lfloor k/2 \rfloor + 1$.
But it is easy to see that the first $\lfloor k/2 \rfloor + 1$ rows of $M$ 
are linearly independent, so equality holds, as claimed.
\end{proof}

This proof was greatly facilitated
by the fact that we did not need to work in a basis,
but were able to find a manageable spanning set of vectors.
% A technical obstacle for applications of the method
% to dynamical systems of greater interest in algebraic combinatorics
% (such as rowmotion of order ideals in posets) 
% arises from the apparent need to identify a convenient basis.

It is worth noting that all the non-unital eigenvalues 
in our example have equal multiplicity.
Of course, the rationality of the entries of the matrix
implies that two Galois-conjugate roots of 1
(i.e., roots that are primitive $m$th roots for the same $m$)
must have the same multiplicity,
but a priori roots that are not Galois-conjugate
could have different multiplicities.
Indeed, consider the rotation action $T_{4,2}$
restricted to multisets in which all elements are distinct;
this set consists of the six elements 01, 02, 03, 12, 13, and 23,
and rotation sends sends 01 to 12 to 23 to 03 to 01
and sends 02 to 13 to 02.
The interested reader can verify that
1 has multiplicity 2, $-1$ has multiplicity 2,
and $i$ and $-i$ each have multiplicity 1,
so in this case the non-unital eigenvalues do not all have the same multiplicity.
% Indeed, the spectral picture for the rotation action
% on $k$-element subsets of $\{0,1,\dots,n-1\}$ (no repetitions allowed)
% is considerably more complicated
% than the picture presented above for multisets (with repetitions allowed).

\end{section}

\begin{section}{Example: Rowmotion on a product of two chains}
\label{sec:rowmotion}

Recall from [S1] the definitions of
a partially ordered set $P$, order ideals (or downsets),
filters (or upsets), and antichains,
and from [S2] and [EP1] the definitions of
the order polytope $\cO(P)$, reverse order polytope $\tilde{\cO}(P)$,
and chain polytope $\cA(P)$ of a poset.
For notation, I remind the reader that
the order polytope of a finite poset $P$
is the set of order-preserving maps $\lambda : P \rightarrow [0,1]$,
viewed as a subset of $\R^{|P|}$ in the natural way.
Put $X = \cO(P)$, $Y = \cA(P)$, $Z = \tilde{\cO}(P)$, and $T = \rowP$.
Piecewise-linear rowmotion $\rowP$ (defined in [EP2])
is a volume-preserving invertible map
from the order polytope to itself
related to Striker and Williams' original definition
of rowmotion for order ideals [SW].
When $P$ is $[a] \times [b]$, $\rowP$ is of order $n := a+b$.
For each $p \in P$, define the evaluation statistic
$1_p(\lambda) = \lambda(p) \in [0,1]$,
so that we have statistics of the form $1_p \circ T^j$
($0 \leq j \leq n-1$, $p \in P$)
spanning a vector space $V$ of dimension at most $nab$.
Let $U$ be the time-evolution operator 
on the space $V$ sending $f \in V$ to $f \circ T \in V$.

Einstein and Propp, in unpublished work 
that was later generalized by Joseph and Roby 
(see sections 4 and 6 of the September 1, 2018 version of [EP1]
and Theorem 5.12 in [JR]), showed that piecewise-linear rowmotion,
although originally defined as a composition 
of piecewise-linear involutions (toggle operators),
can be obtained as a composition of three {\em transfer operators}
$\downt: X \rightarrow Y$, $\upti: Y \rightarrow Z$, 
and $\Theta: Z \rightarrow X$,
mirroring the original definition of rowmotion
in the work of Brouwer and Schrijver [BS].
My goal in this section is to assert 
that linearization is ``functorial''
in the sense that the linearization of piecewise-linear rowmotion
can be expressed as a composition of three linearized transfer operators.
It is not my goal here to prove this assertion 
nor even to make the assertion plausible;
I merely aim to make the {\em content} of the assertion clear
by demonstrating it in the case $a=b=2$
(the simplest case in which the piecewise-linear map $T$ 
is not merely linear or affine).

In the transfer operator formulation, $\rowP$ for an arbitrary finite poset $P$
is expressed as the composition $\Theta \circ \upti \circ \downt$
where $\downt$ ({\em down-transfer}) is a piecewise-linear map from $X$ to $Y$,
$\upti$ ({\em inverse up-transfer}) is a piecewise-linear map from $Y$ to $Z$,
and $\Theta$ ({\em complementation}) is an affine map from $Z$ back to $X$.
In the case $a=b=2$, for $x \in X$ put 
$y = \downt(x)$, $z = \upti(y)$, and $x' = \Theta(z)$ so that $Tx = x'$.
Explicitly,
\begin{eqnarray*}
\downt(x_1,x_2,x_3,x_4) & = &
(x_1 - 0,\:x_2 - x_1,\:x_3 - x_1,\:x_4 - \max(x_2,x_3)), \\
\upti(y_1,y_2,y_3,y_4) & = &
(y_1+\max(y_2,y_3)+y_4,\:y_2 + y_4,\:y_3 + y_4,\:y_4), \ \mbox{and} \\
\Theta(z_1,z_2,z_3,z_4) & = & (1 - z_1,\:1 - z_2,\:1 - z_3,\:1 - z_4).
\end{eqnarray*}

To reduce the complexity of the notation
let us use $x_1,x_2,x_3,x_4$ to denote the coordinate functions
mapping $X$ to $\R$, and likewise for the $y$ and $z$ variables.
The dynamical closure of $x_1,x_2,x_3,x_4$ in $X$
under the action of $\Theta \circ \upti \circ \downt$ 
is 6-dimensional, with additional basis functions
$x_5 = \max(x_2,x_3)$ and the constant function $x_6 = 1$.
The dynamical closure of $y_1,y_2,y_3,y_4$ in $Y$
under the action of $\downt \circ \Theta \circ \upti$ 
is 6-dimensional, with additional basis functions
$y_5 = \max(y_2,y_3)$ and $y_6 = 1$.
The dynamical closure of $z_1,z_2,z_3,z_4$ in $Z$
under the action of $\upti \circ \downt \circ \Theta$ 
is 6-dimensional, with additional basis functions
$z_5 = \max(z_2,z_3)$ and $z_6 = 1$.
With these extra variables we have
% (and at one juncture making use of the fact that
% $\min(y_2,y_3) + \max(y_2,y_3) = y_2 + y_3$)
\begin{eqnarray*}
U_{\downt}(x_1,x_2,x_3,x_4,x_5,x_6) & = &
(x_1,\,x_2\!-\!x_1,\,x_3\!-\!x_1,\,x_4\!-\!x_5,\,x_5\!-\!x_1,\,x_6), \\
U_{\upti}(y_1,y_2,y_3,y_4,y_5,y_6) & = &
(y_1\!+\!y_4\!+\!y_5,\,y_2\!+\!y_4,\,y_3\!+\!y_4,\,y_4,\,y_4\!+\!y_5,\,y_6), \ \mbox{and} \\
U_{\Theta}(z_1,z_2,z_3,z_4,z_5,z_6) & = & 
(z_6\!-\!z_1,\,z_6\!-\!z_2,\,z_6\!-\!z_3,\,z_6\!-\!z_4,\,z_5\!+\!z_6\!-\!z_2\!-\!z_3,\,z_6)
\end{eqnarray*}
which demonstrates linear relations among all the quantities.
(Those who wish to check the formulas should be mindful that
at one point one needs to use the identity $\max(s,t)+\min(s,t) = s+t$.)
Note that $U_{\Theta \circ \upti \circ \downt}
= U_{\downt} \, U_{\upti} \, U_{\Theta}$.

A similar picture prevails for birational rowmotion (see [EP2]),
where now the monoid is multiplicative rather than additive,
along the lines of the example presented in the next section.

% We conjecture that for general $a,b \geq 1$,
% $\dim V = ab + (a-1)(b-1)(a+b-2)/2$ for the action
% of piecewise-linear rowmotion on the order polytope of $[a] \times [b]$,
% and $\dim V_1 = \lceil (a-1)(b-1)/2 \rceil$.

\end{section}

\begin{section}{Example: The Lyness 5-cycle}
\label{sec:lyness}
Here I revisit an example from section 2.6 of [PR].
The Lyness 5-cycle (the smallest nontrivial cluster algebra)
exhibits a non-obvious homomesy:
if $X$ is the set of all $(x,y)$ in $\R^2$ with
$x$, $y$, $x+1$, $y+1$, and $x+y+1$ all nonzero,
and $T:X \rightarrow X$ is the period-5 map
sending $(x,y)$ to $(y,(y+1)/x)$,
and $g: X \rightarrow \R$ is the map
sending $(x,y)$ to $\log |x^{-1} + x^{-2}|$,
then $g$ is 0-mesic.

To fit this into the proposed framework, let us use the multiplicative monoid
generated by the functions $x$, $y$, $x+1$, $y+1$, and $x+y+1$;
technically this is not a vector space
unless we allow fractional exponents,
but for present purposes this turns out not to matter.
The map $T$ sends $x^a y^b (x+1)^c (y+1)^d (x+y+1)^e$
to $x^{a'} y^{b'} (x+1)^{c'} (y+1)^{d'} (x+y+1)^{e'}$
where
$$\left( a' \ b' \ c' \ d' \ e' \right) =
\left(a \ b \ c \ d \ e \right)
\left( \begin{array}{rrrrr}
0 & \ -1 & \ \ \ 0 & \ -1 & \ -1 \\
1 &  0 & 0 &  0 &  0 \\
0 &  0 & 0 &  0 &  1 \\
0 &  1 & 1 &  0 &  1 \\
0 &  0 & 0 &  1 &  0
\end{array} \right) .$$
This 5-by-5 matrix $M$ satisfies $M^5 = I$,
and each of the 5th roots of unity
occurs in the spectrum with multiplicity 1.
% I wrote "Explain 0-mesy" but I don't know what I meant by this

The Lyness 5-cycle is associated with four-row frieze patterns.
The above analysis can also be applied to frieze patterns
with more than four rows; the details differ slightly
according to whether the number of rows is odd or even.
The eigenvalues still have dynamical meaning,
inasmuch as they determine the dimensionalities
of spaces of invariants for powers of the shift-map;
however, the eigenvalues do not appear to be associated
with anything like eigenvectors,
since it is unclear how to to make sense of
expressions like $x^\alpha y^\beta (x+1)^\gamma (y+1)^\delta (x+y+1)^\epsilon$
when $\alpha,\beta,\gamma,\delta,\epsilon$ are algebraic numbers.

\begin{section}{Flatness}

I have already remarked that, for a map $T$ of period $n$,
the rationality of the entries of the presenting matrix
implies that two Galois-conjugate $n$th roots of 1
must have the same multiplicity.  When $n$ is prime,
this implies that all non-unital eigenvalues have the same multiplicity,
but when $n$ is composite,
the function that maps eigenvalues to multiplicities
(hereafter the {\em spectral multiplicity function})
could a priori be far from constant.
It is therefore somewhat surprising that for many examples,
the spectral multiplicity function is ``flat'' in a logarithmic sense,
meaning that its values are bounded by constant multiples of one another.
For instance, in Proposition 4.1
the multiplicity takes on two values whose ratio approaches 1 as $k$ gets large,
while in Proposition 4.2
the multiplicity takes on two values whose ratio approaches 2 as $k$ gets large.

I do not have a quantitative conjecture, but I suggest that this phenomenon --
the flatness of the spectral multiplicity function --
might apply in many situations.
To the extent that this flatness is prevalent, 
it would provide a loose explanation
for the relative paucity of invariants in comparison with homomesies.
For, if the multiplicities of distinct eigenvalues
always have ratio between $1/c$ and $c$ (with $c>1$),
then the multiplicity of 1 divided by the sum of all the multiplicities
must also be between $1/c$ and $c$,
implying that $\dim V / \dim V_1^\perp$ is between $n/c$ and $nc$.
That is, homomesies outnumber invariants roughly by a factor of $n-1$.

\end{section}

\bigskip

\noindent
{\sc Acknowledgments:} The author thanks
Joe Buhler, David Einstein, Darij Grinberg, Michael Joseph, and Tom Roby
for helpful suggestions.

\end{section}

\bigskip

\noindent
{\LARGE References}

\medskip

\noindent
[BS] Andries Brouwer and Alexander Schrijver, 
On the period of an operator, defined on antichains. 
Math. Centrum report ZW 24/74 (1974).

\medskip

\noindent
[EP1] David Einstein and James Propp,
Combinatorial, piecewise-linear, and birational homomesy
for products of two chains.
Preprint;
\href{https://arxiv.org/abs/1310.5294}{https://arxiv.org/abs/1310.5294}.

\medskip

\noindent
[EP2] David Einstein and James Propp,
Combinatorial, piecewise-linear, and birational homomesy 
for products of two chains.
To appear in {\it Algebraic Combinatorics}.

\medskip

\noindent
[JR] Michael Joseph and Tom Roby,
Birational and noncommutative lifts of antichain toggling and rowmotion.
To appear in {\it Algebraic Combinatorics}.

\medskip

\noindent
[PR] James Propp and Tom Roby,
Homomesy in Products of Two Chains.
{\it Electronic Journal of Combinatorics},
Volume 22, Issue 3 (2015), article P3.4.

\medskip

\noindent
[S1] Richard Stanley, Enumerative Combinatorics, vol 1.

\medskip

\noindent
[S2] Richard Stanley, Two Poset Polytopes.
{\em Discrete Comput.\ Geom.\/} {\bf 1} (1986), no.\ 1, 9--23.

\medskip

\noindent
[SW] Jessica Striker and Nathan Williams,
Promotion and rowmotion. 
{\it European Journal of Combinatorics} {\bf 33} (2012), 1919--1942.

\end{document}